\documentclass[leqno]{amsart}
\usepackage{amsmath,amsthm,amssymb,latexsym,graphicx, enumerate, mathtools,mathrsfs,cite}

\newcommand{\NN}{{\mathbb N}}

\newcommand{\RR}{{\mathbb R}}
\newcommand{\TT}{{\mathbb T}}

\newcommand{\I}{\on{I}}
\newcommand{\II}{\on{II}}

\newcommand{\abs}[1]{ \left| #1 \right|}

\newcommand{\del}{\partial}

\newcommand{\DIV}{\on{div}}
\newcommand{\eps}{\varepsilon}

\newcommand{\Id}{\mathrm{Id}}

\newcommand{\norm}[1]{ \left\| #1 \right\| }
\newcommand{\nor}[2]{\left\|#1\right\|_{#2}}
\newcommand{\oline}[1]{\overline{#1}}
\newcommand{\oo}{\infty}
\newcommand{\pars}[1]{\left(#1\right)}

\newcommand{\tr}{\on{tr}}

\newcommand{\mcl}{\mathcal}

\newcommand{\on}{\operatorname}

\newtheorem{lemma}{Lemma}
\newtheorem{proposition}{Proposition}
\newtheorem{theorem}{Theorem}

\theoremstyle{definition}

\numberwithin{equation}{section}
\numberwithin{lemma}{section}
\numberwithin{proposition}{section}
\numberwithin{theorem}{section}
\numberwithin{corollary}{section}
\numberwithin{definition}{section}

\begin{document}
\title[Reduced mean field games]{Dimension reduction techniques in deterministic mean field games}
\author[J.-M. Lasry, P.-L. Lions, B. Seeger]{Jean-Michel Lasry$^{1}$, Pierre-Louis Lions$^{1,2}$, Benjamin Seeger$^{1,3}$}
\address{$^1$Universit\'e Paris-Dauphine \& Coll\`ege de France \\ Place du Mar\'echal de Lattre de Tassigny \\ 75016 Paris, France}
\email{$^2$lions@ceremade.dauphine.fr, $^3$seeger@ceremade.dauphine.fr}


\subjclass[2010]{...}

\date{\today}

\maketitle

\begin{abstract}
	We present examples of equations arising in the theory of mean field games that can be reduced to a system in smaller dimensions. Such examples come up in certain applications, and they can be used as modeling tools to numerically approximate more complicated problems. General conditions that bring about reduction phenomena are presented in both the finite and infinite state-space cases. We also compare solutions of equations with noise with their reduced versions in a small-noise expansion.
\end{abstract}

\section{Introduction}

This paper is concerned with deterministic mean field games (MFG) models in which certain dimension reduction phenomena can be observed. A variety of situations are considered, including both the forward-backward system and the master equation, with state spaces that are finite or infinite. We also study a forward-backward system with strong coupling, often referred to in the literature as ``mean field games of controls,'' and we give some well-posedness results based on the occurrence of the dimension reduction.

We emphasize that the reductions considered in this paper are rather straightforward, and the mathematical developments are kept as simple as possible. The general reduction phenomena considered here can be observed in some specific applications to economics or telecommunications; see for instance \cite{CL}, where such exact reductions are observed in an application to trade crowding. In general, the precise assumptions that give rise to dimension reduction may not be satisfied by more complicated systems, in particular those that incorporate noise. In such situations, the results of this paper suggest ways to build good approximate models with high dimension reduction. This is essential for creating models that are both numerically computable and easily interpretable, thanks to the small number of reduced variables. As a proof of concept, we present an example of a model with noise, for which we prove that the solution is close to a small-noise expansion that exhibits dimension reduction.

The study of mean field games goes back to the works of Lasry and Lions \cite{LL,LLfr1,LLfr2} and Huang, Malham\'e, and Caines \cite{HMC}, and, since then, the field has received a great deal of attention. The models arising in MFG describe differential games between a very large number of indistinguishable players, or agents. When the games are in a Nash equilibrium and the number of agents approaches infinity, the model can be described by a system of two equations: a backward Hamilton-Jacobi-Bellman equation, whose solution is the value function of a representative player; and a forward Fokker-Planck equation that describes the evolution of the population. The mean field games system can be reformulated in terms of the master equation introduced by Lasry and Lions \cite{Lectures}, which is a single equation for the value function that is set on an infinite dimensional space of measures. The merit of this equation is that it can take into account a variety of extensions and more complex models that the forward-backward system fails to encompass. In addition, it is a natural tool to study the infinite-player limit in many situations, including in the presence of common noise. For more details, see Carmona and Delarue \cite{CDmaster} and the book of Cardaliaguet, Delarue, Lasry, and Lions \cite{CDLL}. Alternatively to the PDE approach, mean field games have also been extensively studied from the probabilistic point of view; for a thorough treatment, see the books of Carmona and Delarue \cite{CD1,CD2}.

{\bf Outline of the paper.} For most of the situations considered in this paper, we identify general algebraic conditions that result in a dimension-reduced problem. Such conditions are supported with simple, but representative, examples. The reduced problem can generally be seen to admit a classical solution, which then leads to a solution of the original problem. This is done for a finite state space in Section \ref{S:finite}, and for a continuous state space in Section \ref{S:continuous}. In both sections, both the master equation and the forward-backward system are considered. In Section \ref{S:controls}, we study a forward backward system with strong coupling, and we present a reduced system that consists of a standard forward-backward system coupled with an ordinary differential equation, which admits an existence result. Finally, in Section \ref{S:smallnoise}, we return to the finite state space setting and incorporate a small common noise term, and we justify a formal small noise expansion.

{\bf Notation.} The adjoint of a linear map $A: \RR^m \to \RR^m$ is denoted by $A^*$, and its trace and determinant, respectively, by $\tr A$ and $\det A$. The Euclidean inner product on $x,y \in \RR^m$ is denoted by $\langle x,y \rangle$.

We set
\[
	\mcl P = \mcl P(\RR^d) := \left\{ m : m \text{ is a positive Borel measure on $\RR^d$ and } \int_{\RR^d} m(dx) = 1 \right\},
\]
and, for $q \ge 1$,
\[
	\mcl P_q = \mcl P_q(\RR^d) := \left\{ m \in \mcl P : \int_{\RR^d} |x|^q m(dx) < \oo \right\}.
\]
The convex set $\mcl P_q$ is equipped with the $q$-Wasserstein distance, which, for $m_1,m_2 \in \mcl P_q$, is given by
\[
	d_q(m_1,m_2) := \inf_{\gamma \in \Pi(m_1,m_2) } \pars{ \iint_{\RR^d \times \RR^d} |x-y|^q \gamma(dxdy) }^{1/q},
\]
where $\Pi(m_1,m_2)$ is the set of Borel probability measures on $\RR^d \times \RR^d$ such that $\gamma(A \times \RR^d) = m_1(A)$ and $\gamma(\RR^d \times B) = m_2(B)$ for any Borel subsets $A$ and $B$ of $\RR^d$.

Given $x \in \RR^d$, $\delta_x \in \mcl P$ denotes the Dirac delta measure centered at $x$. For $m \in \mcl P(\RR^d)$ and a Borel measurable map $f: \RR^d \to \RR^n$, $f_\sharp m \in \mcl P(\RR^n)$ is the measure defined, for Borel $A \subset \RR^n$, by
\[
	f_\sharp m(A) := m(f^{-1}(A)).
\]

\section{Finite state space}\label{S:finite}

We first consider games for which there is a finite number $N$ of discrete states, labeled with the index set $\{1,2,\ldots,N\}$. The variables $x^i$ or $X^i$ below represent the concentration of players in the state $i$, while $U^i$ denotes the value function of a typical player in state $i$. We note also that the equations we consider here take the same form as those that arise in later sections as dimension-reduced problems of some continuous-state mean field games models.

We emphasize that the equations studied below are slightly more general that those arising in MFG situations. For more details, see, for instance, \cite{BLL,BLLplanning, Lectures}.

\subsection{The general set-up and the master equation}
In the finite state space case, the master equation is a non-conservative, hyperbolic $N\times N$ system of equations given, for some smooth $F: \RR^N \times \RR^N \to \RR^N$, $G: \RR^N \times \RR^N \to \RR^N$, and $U_0: \RR^N \to \RR^N$, by
\begin{equation}\label{E:finitestate}
	\del_t U + \left[ F(x,U) \cdot \nabla_x \right]U = G(x,U) \quad \text{in } \RR^N \times (0,T], \quad U(0,\cdot) = U_0 \quad \text{in } \RR^N,
\end{equation}
or, coordinate by coordinate, for each $i = 1,2,\ldots, N$,
\[
	\del_t U^i + \sum_{j=1}^N F^j(x,U) \del_{x_j} U^i = G^i(x,U).
\]

We will consider different regimes of well-posedness for the equation \eqref{E:finitestate}. First, the smoothness of $F$, $G$, and $U_0$ are enough to ensure that a unique classical solution of \eqref{E:finitestate} exists for a sufficiently small time horizon $T > 0$. This is constructed with the method of characteristics, as discussed in the next sub-section.

In order for a unique global solution to exist for arbitrary $T > 0$, further structural properties are required, which we take here to be monotonicity.

We say a map $A: \RR^M \to \RR^M$ is monotone if 
\begin{equation}\label{A:monotone}
	\langle A(x) - A(y), x - y \rangle \ge 0 \quad \text{for all } x,y \in \RR^M,
\end{equation}
and is strictly monotone if
\begin{equation}\label{A:monotone}
	\langle A(x) - A(y), x - y \rangle > 0 \quad \text{for all } x,y \in \RR^M \text{ with } x \ne y.
\end{equation}
We will assume that
\begin{equation}\label{A:FGU0monotone}
	\left\{
	\begin{split}
	&(G,F): \RR^{2N} \to \RR^{2N} \quad \text{and} \quad U_0: \RR^N \to \RR^N \quad \text{are monotone, and}\\
	&\text{either $F$ or $G$ is strictly monotone.}
	\end{split}
	\right.
\end{equation}

\subsection{The reduction}
A reduction in dimension will be observed if the various data depend on the distribution $x$ of players only through $n$ reduced quantities, for some integer $n < N$. We represent this with a map $L$ satisfying
\begin{equation}\label{A:linear}
	L: \RR^N \to \RR^n \quad \text{is linear and surjective}.
\end{equation}
We denote by $L^*$ the adjoint of $L$. A consequence of the surjectivity of $L$ is that
\begin{equation}\label{A:linearstar}
	L^*: \RR^n \to \RR^N \quad \text{is injective}.
\end{equation}

We will consider two types of ``reduced'' nonlinearities. First, a map $A: \RR^N \to \RR^N$ is said to completely reduce to $\tilde A: \RR^n \to \RR^n$ if
\begin{equation}\label{A:completereduce}
	A(x) = L^* \tilde{A}(Lx) \quad \text{for all } x \in \RR^N.
\end{equation}
That is, $A$ depends on $x$ only through $Lx$, and, moreover, $Ax$ is perpendicular to the fibers of $L$, since
\[
	\langle A(x), x_0 \rangle = \langle \tilde{A}(Lx), Lx_0 \rangle = 0 \quad \text{for all } x_0 \text{ with } Lx_0 = 0.
\]

We say a map $A: \RR^N \to \RR^N$ fiber-reduces to $\tilde A: \RR^n \to \RR^n$ if
\begin{equation}\label{A:fiberreduce}
	L A(x) = \tilde A(Lx) \quad \text{for all } x \in \RR^N.
\end{equation}
Geometrically, $A$ maps fibers of $L$ to fibers of $L$. That is, \eqref{A:fiberreduce} is equivalent to requiring that, if $Lx_1 = Lx_2$, then $LA(x_1) = LA(x_2)$.

We remark that both concepts of reduction depend on the map $L$, which remains fixed throughout this section.

We shall assume that
\begin{equation}\label{A:FGU0reduce}
	\left\{
	\begin{split}
		&(x \mapsto G(x,U)) \text{ and } U_0 \text{ completely reduce for each fixed $U \in \RR^N$, and}\\
		&(x \mapsto F(x,U)) \text{ fiber-reduces for each fixed $U \in \RR^N$.}
	\end{split}
	\right.
\end{equation}
With a slight relabelling, this means that there exist $\tilde F: \RR^n \times \RR^n \to \RR^n$, $\tilde G: \RR^n \times \RR^n \to \RR^n$, and $\tilde U_0: \RR^n \to \RR^n$ such that, for all $x \in \RR^N$ and $u \in \RR^n$,
\begin{equation}\label{A:FGU0reductionformulae}
	G(x,L^*u) = L^* \tilde G(Lx,u), \quad LF(x,L^*u) = \tilde F(Lx,u), \quad \text{and} \quad U_0(x) = L^* \tilde U_0(L x).
\end{equation}
This leads, formally, to the $n\times n$ system
\begin{equation}\label{E:reducedfinitestate}
	\del_t \tilde U + \left[ \tilde F(y,\tilde U) \cdot \nabla_y \right]\tilde U = \tilde G(y,\tilde U) \quad \text{in } \RR^n \times (0,T], \quad \tilde U(0,\cdot) = \tilde U_0.
\end{equation}

\begin{theorem}\label{T:reducedfinitestate}
	Assume, for some $L$ satisfying \eqref{A:linear}, that $F$, $G$, and $U_0$ satisfy \eqref{A:FGU0monotone} and \eqref{E:reducedfinitestate}. Then
	\begin{equation}\label{tildeGFU0monotone}
		(\tilde G,\tilde F): \RR^{2n} \to \RR^{2n} \quad \text{and} \quad \tilde U_0: \RR^n \to \RR^n \quad \text{are monotone}
	\end{equation}
	and, moreover, there exist unique classical solutions $U$ and $\tilde U$ of respectively \eqref{E:finitestate} and \eqref{E:reducedfinitestate}, which are related by
\begin{equation}\label{finitestateformula}
	U(t,x) = L^* \tilde U(t,Lx).
\end{equation}
\end{theorem}

\begin{proof}
Assume $U_0$ is monotone. Then, for all $x,y \in \RR^N$,
\[
	0 \le \langle U_0(x) - U_0(y) , x-y \rangle = \langle L^*(\tilde U_0(Lx) - \tilde U_0(Ly)), x-y \rangle 
	= \langle \tilde U_0(Lx) - \tilde U_0(Ly),Lx - Ly \rangle.
\]
The monotonicity of $\tilde U_0$ then follows from the surjectivity of $L$.

Now assume that $(G, F)$ is monotone and let $(x_1,u_1),(x_2,u_2) \in \RR^n \times \RR^n$. Then there exist $X_1,X_2 \in \RR^N$ such that $LX_j = x_j$, $j = 1,2$. We then compute
\begin{align*}
	\langle &\tilde G(x_1,u_1) - \tilde G(x_2,u_2),x_1 - x_2 \rangle + \langle \tilde F(x_1,u_1) - \tilde F(x_2,u_2),u_1-u_2 \rangle\\
	&= \langle \tilde G(LX_1,u_1) - \tilde G(LX_2,u_2),L(X_1 - X_2) \rangle + \langle \tilde F(LX_1,u_1) - \tilde F(LX_2,u_2),u_1-u_2 \rangle\\
	&= \langle G(X_1,L^*u_1) - G(X_1,L^* u_2), X_1 - X_2 \rangle + \langle F(X_1,L^*u_1) - F(X_2,L^* u_2), L^*u_1 - L^* u_2 \rangle \ge 0.
\end{align*}

The existence and uniqueness of solutions to both equations is now standard (see \cite{Lectures}), and the formula \eqref{finitestateformula} can be verified with a simple calculation.
\end{proof}

\subsection{The system of characteristics}

The system of characteristics associated to \eqref{E:finitestate} is
\begin{equation}\label{E:finchars}
	\begin{dcases}
		\dot X = F(X,V), &X(0) = x \\
		\dot V = G(X,V), & V(0) = U_0(x),
	\end{dcases}
\end{equation}
which, of course, is analogous to the mean-field games forward-backward system (here written only in forward form for simplicity). The relation to \eqref{E:finitestate} is through the implicit formula
\[
	U(t,X(t)) = V(t) \quad \text{for } t \in [0,T].
\]
The local well-posedness of \eqref{E:finitestate} is a consequence of the fact that, for sufficiently small $T > 0$, $x \mapsto X(x,t)$ is invertible for all $t \in [0,T]$, while a simple computation shows that $X$ is invertible for all $t > 0$ if $(G,F)$ is monotone. 

We now discuss the consequences of the reducibility assumptions for $G$ and $F$ on \eqref{E:finchars}. First, since $F$ fiber-reduces, we find that, if $x_1$ and $x_2$ belong to the same fiber of $L$, then, for all $t \in [0,T]$, $X(t,x_1)$ and $X(t,x_2)$ also belong to the same fiber of $L$; that is, the equation for $X$ can be interpreted as an evolution of fibers of $L$.

On the other hand, the fact that $G$ completely reduces means that the evolution of $V$ depends only on the fiber of $X$. Moreover, $V$ does not have any motion tangential to the fibers of $L$.

Mathematically, the above remarks mean that the system
\begin{equation}\label{E:reducedchars}
	\begin{dcases}
		\dot {\tilde X} = \tilde F(\tilde X,\tilde V), &\tilde X(0) = x \\
		\dot {\tilde V} = \tilde G(\tilde X,\tilde V), & \tilde V(0) = \tilde U_0(x)
	\end{dcases}
\end{equation}
is related to \eqref{E:finchars} by
\[
	\tilde X(t) = LX(t) \quad \text{and} \quad V(t) = L^*\tilde V(t).
\]

\section{Continuous state space}\label{S:continuous}

\subsection{The master equation}

For a given $H: \RR^d \times \mcl P \times \RR^d \to \RR$ and $G: \RR^d \times \mcl P \to \RR$, we study the master equation
\begin{equation}\label{E:master}
	\left\{
	\begin{split}
	&- \frac{\del U}{\del t} + H(x, m, D_x U) + \int_{\RR^d} D_m U(t,x,m,y) \cdot D_p H(y,m, D_y U(t,y,m))\;m(dy) = 0 \\
	& \qquad \qquad \qquad \qquad \qquad \qquad \qquad \qquad \qquad \text{in } (0,T) \times \RR^d \times \mcl P \text{ and} \\
	&U(T,x,m) = G(x,m)  \qquad \text{in } \RR^d \times \mcl P.
	\end{split}
	\right.
\end{equation}

The analogue of the linear map $L$ from the previous section is given, for some $\phi: \RR^d \to \RR^m$, by
\[
	\mcl P \ni m \mapsto \int_{\RR^d} \phi(y) m(dy) \in \RR^m.
\]
Let us assume that
\begin{equation}\label{A:phi}
	\left\{
	\begin{split}
	&\phi \in C^1(\RR^d,\RR^m), \quad |\phi(x)| \le C(1 + |x|^K) \quad \text{for all } x \in \RR^d \text{ and some } C > 0 \text{ and } K > 0, \text{ and}\\
	&\del \mcl C \subset \phi(\RR^d), \text{ where}\\
	&\mcl C := \left\{ z \in \RR^m : z = \int_{\RR^d} \phi(x) m(dx) \quad \text{for some } m \in \mcl P_K \right\}.
	\end{split}
	\right.
\end{equation}
We note that \eqref{A:phi} implies that the convex set $\mcl C$ is closed, since, for any $x \in \RR^d$,
\[
	\phi(x) = \int_{\RR^d} \phi(y) \delta_x(dy) \in \mcl C.
\]

We impose the following algebraic conditions that reflect the fact that the dependence of \eqref{E:master} on the measure variable $m$ is felt only through the quantity $\int \phi dm$:
\begin{equation}\label{A:Hreduces}
	\left\{
	\begin{split}
		&\text{there exists $h: \mcl C \times \RR^m \to \RR^m$ such that}\\
		&H(x,m,D\phi(x) \cdot \psi) = \phi(x) \cdot h\pars{\int_{\RR^d} \phi(y) m(dy), \psi} \text{ for all } (x,m,\psi) \in \RR^d \times \mcl P_K \times \RR^m,
	\end{split}
	\right.
\end{equation}
and
\begin{equation}\label{A:Greduces}
	\left\{
	\begin{split}
	&\text{for some $g: \mcl C \to \RR^m$ and for all $(x,m) \in \RR^d \times \mcl P_K$},\\
	&G(x,m) = \phi(x) \cdot g\pars{\int_{\RR^d} \phi(y)m(dy)}.
	\end{split}
	\right.
\end{equation}
The dimension-reduced problem is then given, for some $f: \del \mcl C \times [0,T) \to \RR^m$, by the following boundary-terminal-value problem
\begin{equation}\label{E:masterreduced}
\begin{dcases}
	-\del_t u + h(z,u) + z \cdot \del_u h(z,u) \del_z u = 0&\text{in } [0,T) \times \mcl C, \\
	u(T,z) = g(z) & \text{in } \mcl C, \text{ and} \\
	u(t,z) = f(t,z) & \text{on } \del \mcl C \times [0,T).
\end{dcases}
\end{equation}

We now impose conditions on $g$ and $h$, which are related to monotonicity, that make \eqref{E:masterreduced}, and, hence, \eqref{E:master}, into a well-posed problem. The conditions on the boundary value $f$ will be much more restrictive, and, in fact, \eqref{E:masterreduced} is well-posed for only one choice of $f$.

We assume that
\begin{equation}\label{A:hmonotone}
	(z,u) \mapsto (-h(z,u), z \cdot h_u(z,u)) \quad \text{is monotone},
\end{equation}
and
\begin{equation}\label{A:gmonotone}
	g \text{ is strictly monotone in } \mcl C.
\end{equation}

\begin{theorem}\label{T:masterreduced}
	Assume \eqref{A:phi}, \eqref{A:Hreduces}, \eqref{A:Greduces}, \eqref{A:hmonotone}, and \eqref{A:gmonotone}. Then there exists a unique $f \in C^1(\del \mcl C \times [0,T])$ such that \eqref{E:masterreduced} admits a unique classical solution $u$. Moreover, the formula
	\begin{equation}\label{formula}
		U(t,x,m) = \phi(x) u\pars{ t, \int_{\TT^d} \phi(y) m(dy)} \quad \text{for } (t,x,m) \in [0,T] \times \RR^d \times \mcl P_K
	\end{equation}
	defines a classical solution of \eqref{E:master}.
\end{theorem}

\begin{proof} 
A simple consequence of \eqref{A:Hreduces} is
\[
	D\phi(x) \cdot D_p H(x,m,D\phi(x) \psi) = \phi(x) \del_\psi h\pars{\int_{\RR^d} \phi(y)m(dy),\psi} \quad \text{for all } (x,m,\psi) \in \RR^d \times \mcl P_K \times \RR.
\]

The unique solution of \eqref{E:masterreduced} is constructed using the method of characteristics, which is the system of ordinary differential equations given, for some fixed $z \in \mcl C$, by
\begin{equation}\label{E:chars}
	\begin{dcases}
		\dot{\mcl Z} = - \mcl Z h_u(\mcl Z,\mcl U), & \mcl Z(T,z) = z, \\
		\dot{\mcl U} = h(\mcl Z, \mcl U), & \mcl U(T,z) = g(z).
	\end{dcases}
\end{equation}
As can be checked, the assumptions \eqref{A:hmonotone} and \eqref{A:gmonotone} imply that the map $z \mapsto \mcl Z(z,t)$ is strictly monotone for each $t \in [0,T]$. 

We next claim that, if $z \in \del \mcl C$, then $\mcl Z(z,t) \subset \del \mcl C$ for all $t \in [0,T]$. To see this, let $x_0 \in \RR^d$ be such that $\phi(x_0) = z$. Then \eqref{A:Hreduces} implies that
\[
	H(x_0, \delta_{x_0}, D\phi(x_0) \cdot v) = z \cdot h(z,v) \quad \text{for all } v \in \RR^m.
\]
Let $n$ be the normal vector to $\del \mcl C$ at $z$ and let $t \in \RR$. Then $D\phi(x_0) \cdot n = 0$, so, setting $v = u + t n$ for some $u \in \RR^m$,
\[
	H(x_0, \delta_{x_0}, D\phi(x_0) \cdot u) = z \cdot h(z, u + t n).
\]
Differentiating in $t$ and setting $t = 0$ yields
\[
	z \cdot h_u(z,u) \cdot n = 0.
\]
In particular, this implies that $\dot {\mcl Z}$ is tangential to $\del \mcl C$ whenever $\mcl Z \in \del \mcl C$, which yields the claim.

From this and the monotonicity of $\mcl Z$, it follows that $\mcl Z(\cdot, t) : \mcl C \to \mcl C$ is invertible, and therefore we can implicitly define
\[
	u(t,\mcl Z(t,z)) = \mcl U(t,z) \quad \text{for }(t,z) \in [0,T] \times \mcl C,
\]
which is then the unique classical solution of \eqref{E:masterreduced}, as long as
\[
	\del_t f(t,z) = h(z, f(t,z)) \quad \text{for } (t,z) \in [0,T] \times \del \mcl C \quad \text{and} \quad f(T,z) = g(z).
\]
The fact that \eqref{formula} defines a classical solution of \eqref{E:master} follows by calculation. 

\end{proof}

\subsection{The forward-backward system}
The forward-backward associated to \eqref{E:master} is
\begin{equation}\label{E:fbsystem}
	\begin{dcases}
		-u_t + H(x,m, Du) = 0, & u(T,\cdot) = G(\cdot,m(T)) \\
		m_t - \DIV(m D_p H(x,m,Du)) = 0, & m(0) = m_0,
	\end{dcases}
\end{equation}
and the reduced version of \eqref{E:fbsystem} is the system of ordinary differential equations
\begin{equation}\label{E:generalmdreduced}
	\begin{dcases}
	-\dot \psi + h(z,\psi), & \psi(T) = g(z(T))\\
	\dot z + z \cdot h_u(z,\psi) = 0, & z(0) = z_0.
	\end{dcases}
\end{equation}
We note that this system shares a connection with \eqref{E:chars}, but \eqref{E:generalmdreduced} is given in a forward-backward form.

The following result is immediate.
\begin{theorem}\label{T:fbsystem}
	Assume \eqref{A:phi}, \eqref{A:Hreduces}, \eqref{A:Greduces}, \eqref{A:hmonotone}, and \eqref{A:gmonotone}. Then \eqref{E:generalmdreduced} has a unique solution $(\psi,z): [0,T] \to \RR \times \mcl C$ for every $z_0 \in \mcl C$. Moreover, if $u(t,x) = \phi(x) \cdot \psi(t)$ and $m$ is the solution of
	\[
		m_t - \DIV\left[ m D_p H(x,m,Du) \right] = 0 \quad \text{in } \RR^d \times (0,T] \quad \text{and} \quad m(\cdot,0) = m_0,
	\]
	then $(u,m)$ is a solution of \eqref{E:fbsystem}, and
	\[	
		z(t) = \int_{\RR^d} \phi(y) m(y,t)dy.
	\]
\end{theorem}

\subsection{Examples}

We present here some examples to which the theory of the previous results can be applied. In these, $\phi$ will have a power-like structure.

We first look at a one-dimensional example; that is, $m = 1$. Define, for some functions $a: [0,\oo) \to [0,\oo)$, $b: [0,\oo) \to \RR$, and $c: [0,\oo) \to \RR$, and some $q \ge 2$,
\[ 
	H(x,m,p) = \frac{1}{q'} \left[ \frac{1}{q} a(z)|p|^q + b(z) (p\cdot x) + c(z) |x|^{q'}\right] \quad \text{for } z = \frac{1}{q'} \int_{\TT^d} |y|^{q'} m(dy).
\]
Then \eqref{A:phi} and \eqref{A:Hreduces} hold with
\[
	\phi(y) = \frac{1}{q'} |y|^{q'} \quad \text{and} \quad h(z,u) = \frac{1}{q} a(z) |u|^q + b(z) u + c(z).
\]
Note that, in this example, $\mcl C = [0,\oo)$.

It remains to check \eqref{A:hmonotone}, which, when $m = 1$, is equivalent to
\begin{equation}\label{A:hconditions}
	\frac{\del h}{\del z} \le 0, \quad \frac{\del^2 h}{\del u^2} \ge 0, \quad \text{and} \quad z\pars{ \frac{\del^2 h}{\del u \del z}}^2 \le - 4 \frac{\del h}{\del z} \frac{\del^2 h}{\del u^2}.
\end{equation}

It can be seen from some tedious but straightforward calculations that \eqref{A:hconditions} holds if and only if
\begin{equation}\label{abc}
	\left\{
	\begin{split}
	&a > 0, \quad a' \le 0, \quad c' \le 0, \quad z \mapsto a(z) z^{\frac{4(q-1)}{q}}\text{ is nondecreasing,}\\
	&b \text{ is constant if } q > 2, \text{ and}\\
	&b'(z)^2 \le 2 a'(z)c'(z) \text{ if } q = 2.
	\end{split}
	\right.
\end{equation}
We remark that one can also come up with sufficient conditions on $a$, $b$, and $c$ to make \eqref{A:hconditions} hold in the case where $q < 2$ (of course, $a$ and $b$ constant and $c$ nonincreasing always works), but deriving the necessary conditions is considerably harder.

We now consider some quadratic type examples, allowing for $m > 1$. For simplicity of presentation, we take $d =1$, although the idea can be generalized to higher dimensions.

We set
\[
	H(x,m,p) = \frac{1}{2}p^2 - f_0(z) - xf_1(z) - \frac{x^2}{2} f_2(z) \quad \text{where} \quad z = (z_1,z_2) = \int_{\RR} \pars{ y,\frac{y^2}{2}}m(dy),
\]
and then \eqref{A:phi} and \eqref{A:Hreduces} hold with
\[
	\phi(y) = \pars{1, y, \frac{y^2}{2}} \quad \text{and} \quad h(z,u) = \pars{ \frac{1}{2} u_1^2 - f_0(z), u_1u_2 - f_1(z), u_2^2 - f_2(z) }.
\]
Here, $u = (u_0,u_1,u_2) \subset \RR^3$ and $z = (z_0,z_1,z_2)$ belongs to the set
\[
	\mcl C = \{1\} \times \left\{ (z_1,z_2) : \frac{1}{2} z_1^2 \le z_2 \right\}.
\]
We claim that $h$ satisfies \eqref{A:hmonotone} as long as $f$ is monotone. Indeed, for $(z,u), (\tilde z, \tilde u) \in \mcl C \times \RR^3$, we compute
\begin{align*}
	-\pars{ h(z,u) - h(\tilde z, \tilde u)} &+ \pars{ z \cdot h_u(z,u) - \tilde z \cdot h_u(\tilde z, \tilde u)}
	= \langle f(z) - f(\tilde z), z - \tilde z \rangle \\
	&+ \frac{1}{2}(u_1 - \tilde u_1)^2 (z_0 + \tilde z_0) + \frac{1}{2}(u_1 - \tilde u_1)(u_2 - \tilde u_2)(z_1 + \tilde z_1) + (u_2 - \tilde u_2)^2 (z_2 + \tilde z_2)\\
	&\ge \frac{3}{4} (u_1 - \tilde u_1)^2 + \frac{1}{8} \pars{ 2z_1(u_2 - \tilde u_2) + u_1 - \tilde u_1}^2 + \frac{1}{8} \pars{ 2\tilde z_1(u_2 - \tilde u_2) + u_1 - \tilde u_1}^2 \ge 0,
\end{align*}
where we have used the fact that $z_0 = \tilde z_0 = 1$, $z_2 \ge z_1^2/2$, and $\tilde z_2 \ge \tilde z_1^2/2$.

\section{Strongly coupled mean field games}\label{S:controls}

We now turn to mean-field games systems with strong coupling, in which the mean field interactions of the infinitesimal players depend not only on their empirical distribution, but also on the distribution of the controls. 

More precisely, we study systems of the form
\begin{equation}\label{E:controls}
	\begin{dcases}
	-\frac{\del u}{\del t} + H(x,Du,m,\mu) = 0 \quad \text{in } \RR^d \times [0,T], & u(\cdot,T) = F(x,m(\cdot,T)) \\
	\frac{\del m}{\del t} - \DIV\left[ D_p H(x,Du,m,\mu)m \right] = 0 \quad \text{in } \RR^d \times [0,T], & m(\cdot,0) = m_0,\\
	\mu := \pars{ \Id, -D_p H(x,Du,m)}_\sharp m,
	\end{dcases}
\end{equation}
the third condition meaning that, for a bounded continuous function $\Psi: \RR^d \times \RR^d$,
\[
	\int_{\RR^d \times \RR^d} \Psi(x, v)\; d\mu(x,v) = \int_{\RR^d} \Psi(x, -D_p H(x,Du_t(x), m_t))\;m(dx).
\]

Such systems, which are also known in the literature as ``mean field games of controls,'' ``extended mean field games,'' or ``mean field games with interacting controls,'' are very natural from the standpoint of applications, see for instance \cite{CL}. For various existence and uniqueness results, see, for instance, \cite{GV1,GV2,GPV, BLL,K}.

\subsection{A reduced system} 
We now demonstrate that, with some structural assumptions on the Hamiltonian that can be readily verified in some applications, the system \eqref{E:controls} reduces to a simpler one in which the evolution of $\mu$ is reduced to an ordinary differential equation. We present the formal structural computations here, and in the next subsection we present a simple existence result.

For $H: \RR^d \times \RR^d \times \RR^m \to \RR$, $G: \RR^d \times \mcl P \to \RR$, $m_0 \in \mcl P$, and $\Phi : [0,T] \times \RR^d \times \RR^d  \to \RR^m$, we consider the system
\begin{equation}\label{E:controlseg}
	\begin{dcases}
	-\frac{\del u}{\del t} + H(x,Du,\phi) = 0 \quad \text{in } \RR^d \times [0,T], & u(T,\cdot) = G(x,m(T,\cdot)), \\
	\frac{\del m}{\del t} - \DIV\left[ D_p H(x,Du,\phi)m \right] = 0 \quad \text{in } \RR^d \times [0,T], & m(0,\cdot) = m_0,\\
	\phi(t) := \int_{\RR^d} \Phi(t,y,Du(t,y))m_t(dy).
	\end{dcases}
\end{equation}
Note that, if $H$ is strictly convex in $p$, then $p \mapsto D_p H(\cdot,p,\cdot)$ is invertible, so that $\phi$ is indeed some functional of $\mu$.

The main structural assumption we make is that
\begin{equation}\label{A:controlreduction}
	\left\{
	\begin{split}
		&\text{there exist } A \in C([0,T] \times \RR^m , \RR^{m \times m}) \text{ and } B \in C([0,T] \times \RR^m , \RR^m) \text{ such that,}\\
		&\text{for all } (t,x,p,\phi) \in [0,T] \times \RR^d \times \RR^d \times \RR^m,\\
		&\frac{\del \Phi}{\del t}(t,x,p) - D_p H(x,p,\phi) \cdot D_x \Phi(t,x,p) + D_x H(x,p,\phi) \cdot D_p\Phi(t,x,p) \\
		&\quad+ A(t,\phi)\Phi(t,x,p) + B(t,\phi) = 0.
	\end{split}
	\right.
\end{equation}
Define $f: [0,T] \times \RR^m \to \RR^m$ by
\begin{equation}\label{f}
	f(t,\phi) := A(t,\phi)\phi + B(t,\phi) \quad \text{for } (t,\phi) \in [0,T] \times \RR^m
\end{equation}
and $g: \mcl P \to \RR^m$ by
\begin{equation}\label{g}
	g(m) =  \int_{\RR^d} \Phi(T,y, D_yG(y,m) ) m(dy) \quad \text{for all } m \in \mcl P.
\end{equation}
We now introduce a reduced system, which is nothing more than a ``standard'' mean-field games system coupled with an ordinary differential equation:
\begin{equation}\label{E:reducedcontrols}
	\begin{dcases}
	-\frac{\del u}{\del t} + H(x,Du,\phi) = 0 \quad \text{in } \RR^d \times [0,T], & u(T,\cdot) = G(x,m(\cdot,T)) \\
	\frac{\del m}{\del t} - \DIV\left[ D_p H(x,Du,\phi)m \right] = 0 \quad \text{in } \RR^d \times [0,T], & m(0,\cdot) = m_0,\\
	-\dot \phi(t) = f(t,\phi(t)), & \phi(T) = g(m(\cdot,T)).
	\end{dcases}
\end{equation}

\begin{proposition}\label{P:reducedcontrols}
	Assume \eqref{A:controlreduction}. Then the triple $(u,m,\phi)$ is a classical solution of \eqref{E:controlseg} if and only if it is a classical solution of \eqref{E:reducedcontrols}.
\end{proposition}

\begin{proof}
	If $(u,m,\phi)$ solves \eqref{E:controlseg}, then a straightforward calculation and \eqref{A:controlreduction} imply that $\phi$ solves the equation in \eqref{E:reducedcontrols}. Conversely, assume that $(u,m,\phi)$ solves \eqref{E:reducedcontrols} and set
	\[
		\psi(t) := \int_{\RR^d} \Phi(t,y,Du(t,y))m(t,y)dy.
	\]
	Then both $\phi$ and $\psi$ solve the terminal value problem
	\[
		- \dot X(t) = a(t) X(t) + b(t), \quad X(T) = g(m(T,\cdot)),
	\]
	where
	\[
		a(t) := A(t,\phi(t)) \quad \text{and} \quad b(t) := B(t,\phi(t)) \quad \text{for } t \in [0,T],
	\]
	from which we conclude that $\phi = \psi$, and therefore $(u,m,\phi)$ solves \eqref{E:controlseg}.
\end{proof}

\subsection{Existence of solutions}

We demonstrate the existence of a solution of \eqref{E:reducedcontrols}, which, by virtue of Proposition \ref{P:reducedcontrols}, gives rise to a solution of \eqref{E:controlseg}. 

Based on examples that we discuss later, the growth estimates for the various data will be in terms of some powers
\begin{equation}\label{A:powers}
	q > 1, \quad r \ge q, \quad \text{and} \quad \gamma := \frac{r}{q-1} \ge q'.
\end{equation}
We shall assume
\begin{equation}\label{A:Phi}
	\left\{
	\begin{split}
		&\text{there exists $C > 0$ such that, for all }(t,y,p) \in [0,T] \times \RR^d \times \RR^d,\\
		&|\Phi(t,y,p)| \le C(1 + |p|^r), \quad |D_p \Phi(t,y,p)| \le C(1 + |p|^{r-1}), \text{ and}\\
		&|D_x \Phi(t,y,p)| \le C(1 + |p|^{r - q + 1}),
	\end{split}
	\right.
\end{equation}
\begin{equation}\label{A:m0moment}
	m_0 \in \mcl P_\lambda \quad \text{for some } \lambda > \gamma,
\end{equation}
\begin{equation}\label{A:Hgrowth}
	\left\{
	\begin{split}
	&H \in C^{1,1}
	(\RR^d \times \RR^d \times \RR^m) \quad \text{is convex in the gradient variable,}\\
	&\text{for some $C > 0$ and for all $(p,x,\phi) \in \RR^d \times \RR^d \times \RR^m$,}\\
	&|H(p,x)| + |D_x H(x,p,\phi)| \le C(1 + |p|^q),  \quad |D_p H(x,p,\phi)| \le C(1 + |p|^{q-1}),\\
	&D^2_{xx} H(x,p,\phi) \le 0, \quad \text{and} \quad |D^2_{px} H(x,p,\phi)| \le C,
	\end{split}
	\right.
\end{equation}
\begin{equation}\label{A:Ggrowth}
	\left\{
	\begin{split}
		&\text{for all $m \in \mcl P_\gamma$, } G(\cdot,m) \in C^{1,1}(\RR^d) \text{ is convex, and, for all $x \in \RR^d$ and some $C > 0$},\\
		&|G(x,m)| \le C(1 + |x|^{q'}), \quad |D_x G(x,m)| \le C(1 + |x|^{q'-1}), \quad \text{and} \quad |D^2_x G(x,m)| \le C(1+|x|)^{q'-2},
	\end{split}
	\right.
\end{equation}
and
\begin{equation}\label{A:Gmdep}
	\left\{
	\begin{split}
		&\text{there exists a modulus $\omega: [0,\oo) \to [0,\oo)$ such that, for all $m_1,m_2 \in \mcl P_\gamma$},\\
		&|G(x,m_1) - G(x,m_2)| + |D_x G(x,m_1) - D_x G(x, m_2)| \le C(1+|x|^{q'}) \omega(d_\gamma(m_1, m_2)).
	\end{split}
	\right.
\end{equation}
With regards to the coefficients $A$ and $B$ from \eqref{A:controlreduction}, we will need to assume
\begin{equation}\label{A:AandB}
	\sup_{t \in [0,T]} \pars{ \nor{A(t,\cdot)}{C^{0,1}(\RR^d)} + \nor{B(t,\cdot)}{C^{0,1}(\RR^d) } } < \oo.
\end{equation}

We first present, without proof, some standard results on the solvability of the Hamilton-Jacobi, continuity, and ordinary differential equations.

\begin{lemma}\label{L:classical}
	\begin{enumerate}[(a)]
	\item\label{L:HJ} Assume that $H$ satisfies \eqref{A:Hgrowth}, $\phi \in C^1([0,T])$, and $\tilde G \in C^{1,1}(\RR^d)$ satisfies, for some constant $C > 0$ and all (almost all for the last inequality) $x \in \RR^d$,
	\[
		|\tilde G(x)| \le C(1 + |x|^{q'}), \quad |D_x \tilde G(x)| \le C(1 + |x|^{q'-1}), \quad \text{and} \quad |D^2_x \tilde G(x)| \le C(1+|x|)^{q'-2}.
	\]
	Then there exists a unique solution $u \in C^{1,1}(\RR^d \times [0,T])$ of the terminal value problem
	\[
		-u_t + H(x,Du,\phi) = 0 \quad \text{in } \RR^d \times [0,T) \quad \text{and} \quad u(T,\cdot) = \tilde G \quad \text{on } \RR^d
	\]
	which, for some $C > 0$ depending on the bounds in \eqref{A:Hgrowth} and the bounds for $\tilde G$, satisfies the bounds
	\begin{equation}\label{ubounds}
	\sup_{(t,x) \in [0,T] \times \RR^d} \pars{ \frac{|u(t,x)|}{1 + |x|^{q'}} + \frac{|Du(t,x)|}{1 + |x|^{q'-1}} + \frac{D^2 u(x)}{(1+|x|)^{q'-2}}} \le C \quad \text{and} \quad D^2 u(t,x) \ge 0.
\end{equation}
	\item\label{L:continuity} Assume that $m_0$ satisfies \eqref{A:m0moment} and $b \in C([0,T],C^1(\RR^d))$ satisfies
	\[
		\sup_{t \in [0,T]} \pars{ \frac{|b(t,x)|}{1 + |x|} + |D_xb(t,x)| } < \oo.
	\]
	Then there exists a unique classical solution of the continuity equation
	\[
		\frac{\del m}{\del t} - \DIV[b(t,x)m] = 0 \quad \text{in } \RR^d \times (0,T] \quad \text{and} \quad m(0,\cdot) = m_0 \quad \text{in } \RR^d,
	\]
	which, for some $C > 0$ depending on $\int_{\RR^d} |x|^\lambda m_0(dx)$ and the bounds for $b$, satisfies
	\begin{equation}\label{mbound}
		\max_{t \in [0,T]} \int_{\RR^d} |x|^{\lambda}m_t(dx) \le C.
	\end{equation}
	\item\label{L:ODE} Assume \eqref{A:AandB}, let $f$ be defined by \eqref{f}, and let $\tilde g \in \RR$. Then there exists a unique solution of
	\[
		-\dot\phi(t) = f(t,\phi(t)) \quad \text{for } t \in [0,T) \quad \text{and} \quad \phi(T) = \tilde g
	\]
	and, for some constant $C > 0$ depending only on the bound for \eqref{A:AandB},
	\[
		\max_{t \in [0,T]} |\phi(t)| \le C|\tilde g|.
	\]
	\end{enumerate}
\end{lemma}

The next result gives a quantitative estimate for the modulus of continuity of the function $g$ defined by \eqref{g}. 

\begin{lemma}\label{L:g}
	Assume that $g$ is given by \eqref{g}. Then, for some constant $C > 0$ depending only on the bounds for $G$ and $\Phi$ in \eqref{A:Phi}, \eqref{A:Ggrowth}, and \eqref{A:Gmdep}, 
	\[
		\abs{ g(m_1) - g(m_2)} \le C\pars{ 1 + \int_{\RR^d} |x|^\gamma(m_1(x) + m_2(x))dx} \omega(d_\gamma(m_1,m_2)) \quad \text{for all } m_1,m_2 \in \mcl P_\gamma,
	\]
	where $\omega$ is the modulus from \eqref{A:Gmdep}.
\end{lemma}

In the proof below, the constant $C > 0$, which depends only on the bounds in the given assumptions, may change from line to line.

\begin{proof}[Proof of Lemma \ref{L:g}]
	We first write
	\[
		g(m_1) - g(m_2) = \I + \II,
	\]
	where
	\[
		\I := \int_{\RR^d} \Phi(T,y,D_y G(y,m_1)) (m_1(y) - m_2(y))dy
	\]
	and
	\[
		\II := \int_{\RR^d} \pars{ \Phi(T,y,D_yG(y,m_1)) - \Phi(T,y,D_y G(y,m_2))}m_2(y)dy.
	\]
	We then fix $\pi \in \Pi(m_1,m_2)$ and estimate
	\begin{align*}
		\abs{ \I } &\le \abs{  \int_{\RR^d} \Phi(T,x,D_y G(x,m_1))m_1(x)dx - \int_{\RR^d} \Phi(T,y,D_y G(y,m_1))m_2(y)dy }\\
		&\le \iint_{\RR^d \times \RR^d} |\Phi(T,x,D_y G(x,m_1)) -  \Phi(T,y,D_y G(y,m_1))| \pi(dxdy)\\
		&\le C \iint_{\RR^d \times \RR^d} (1 + |D_x G(x,m_1)|^{r - q + 1} + |D_y G(y,m_1)|^{r - q + 1})|x-y| \pi(dxdy)\\
		&+ C \iint_{\RR^d \times \RR^d} (1 + |D_x G(x,m_1)|^{r-1} + |D_y G(y,m_1)|^{r-1} )|D_x G(x,m_1) - D_y G(y,m_1| \pi(dxdy) \\
		&\le C\iint_{\RR^d \times \RR^d} (1 + |x|^{\gamma - 1} + |y|^{\gamma - 1})|x-y| \pi(dxdy) \\
		&\le C\pars{ 1 + \int_{\RR^d} |x|^\gamma m_1(dx) + \int_{\RR^d} |y|^\gamma m_2(dy)}^{\frac{\gamma - 1}{\gamma} }\pars{ \int_{\RR^d \times \RR^d} |x-y|^\gamma \pi(dxdy) }^{1/\gamma},
	\end{align*}
	so that taking the infimum over $\pi \in \Pi(m_1,m_2)$ gives
	\[
		\abs{ \I} \le C\pars{ 1 + \int_{\RR^d} |x|^\gamma m_1(dx) + \int_{\RR^d} |y|^\gamma m_2(dy)}^{\frac{\gamma - 1}{\gamma} } d_\gamma(m_1,m_2).
	\]
	The estimate is complete upon computing
	\begin{align*}
		\abs{ \II} &\le C\int_{\RR^d} (1 + |D_y F(y,m_1)|^{r-1} + |D_y F(y,m_2)|^{r-1})|D_y F(y,m_1) - D_y F(y,m_2)|m_2(dy) \\
		&\le C \int_{\RR^d} (1 + |y|^\gamma)m_2(dy) \omega(d_\gamma(m_1,m_2)).
	\end{align*}
\end{proof}

We now introduce the fixed point problem that will yield a solution of \eqref{E:reducedcontrols}. We define a map $\mcl T : \mcl P_\gamma \to \mcl P_\gamma$ in two steps as follows. First, given $\oline{m} \in \mcl P_\gamma$, we solve the two terminal value problems
\begin{equation}\label{E:givenm}
	\begin{dcases}
	-\frac{\del u}{\del t} + H(x,Du,\phi) = 0 \quad \text{in } \RR^d \times [0,T], & u(T,\cdot) = G(x,\oline{m}) \\
	-\dot \phi(t) = f(t,\phi(t)) \quad \text{in } \RR^d \times [0,T], & \phi(T) = g(\oline{m}).
	\end{dcases}
\end{equation}
More precisely, the ordinary differential equation is solved first, and its solution $\phi$ is then fed into the Hamilton-Jacobi equation. We then solve the continuity equation
\begin{equation}\label{E:givenuphi}
	\frac{\del m}{\del t} - \DIV\left[ D_p H(x,Du,\phi)m \right] = 0, \quad m(\cdot,0) = m_0,
\end{equation}
and we set $\mcl T \oline{m} := m(T)$.

Let us first check that this map is well-defined.

\begin{lemma}\label{L:goodT}
	Assume \eqref{A:powers} - \eqref{A:AandB}. Then the map $\mcl T$ defined through \eqref{E:givenm} and \eqref{E:givenuphi} is well-defined and continuous from $\mcl P_\gamma$ to $\mcl P_\lambda$, and, moreover, there exists a constant $C > 0$ depending only on the various bounds in \eqref{A:powers} - \eqref{A:AandB} such that, for all $m \in \mcl P_\gamma$,
	\begin{equation}\label{Tcompact}
		\int_{\RR^d} |x|^\lambda (\mcl Tm)(dx) \le C.
	\end{equation}
\end{lemma}

\begin{proof}
	Lemma \ref{L:classical}\eqref{L:ODE} implies that the ordinary differential equation in \eqref{E:givenm} has a unique, continuously differentiable solution. Then, by \eqref{A:Ggrowth} and Lemma \ref{L:classical}\eqref{L:HJ}, the Hamilton-Jacobi equation has a unique classical $u$ satisfying \eqref{ubounds}.
	
	Now set
	\[
		b(t,x) := D_p H(x,Du(t,x), \phi(t)) \quad \text{for } (t,x) \in [0,T] \times \RR^d.
	\]
	Then, by \eqref{A:Hgrowth} and \eqref{ubounds}, 
	\[
		|b(t,x)| \le C\pars{ 1 + |Du(t,x)|^{q-1}} \le C(1 + |x|)
	\]
	and
	\[
		D_x b(t,x) = D_{px} H(x,Du(t,x),\phi(t)) + D^2_{pp} H(x,Du(t,x),\phi(t))D^2u(t,x),
	\]
	and so
	\[
		|D_x b(t,x)| \le C. 
	\]
	Then Lemma \ref{L:classical}\eqref{L:continuity} implies that \eqref{E:givenuphi} admits a unique classical solution with $m(t) \in \mcl P_\lambda$ for all $t \in[0,T]$, and, moreover, the bound in \eqref{Tcompact} is satisfied.
				
	We now turn to the continuity of the map $\mcl T$. Let $(\oline{m}_n)_{n=1}^\NN \subset \mcl P_\gamma$ and $\oline{m} \in \mcl P_\gamma$ be such that
	\[
		\lim_{n \to \oo} d_\gamma(\oline{m}_n, \oline{m}) = 0.
	\]
	For $n \in \NN$, let $(u_n,\phi_n)$ and $(u,\phi)$ be the solutions of \eqref{E:givenm} corresponding to respectively $\oline{m}_n$ and $\oline{m}$, and let $m_n$ and $m$ be the corresponding solutions of \eqref{E:givenuphi}.
	
	In view of the fact that
	\[
		\sup_{n \in \NN} \int_{\RR^d} |x|^\gamma \oline{m}_n(dx) < \oo,
	\]
	we have, by Lemma \ref{L:g}, 
	\[
		\lim_{n \to \oo} g(\oline{m}_n) = g(\oline{m}).
	\]
	Also, \eqref{A:Gmdep} implies that, as $n \to \oo$, $G(\cdot,\oline{m}_n)$ and $D_x G(\cdot,\oline{m}_n)$ converge locally uniformly to respectively $G(\cdot,\oline{m})$ and $D_x G(\cdot,\oline{m})$. Therefore, we also have that, as $n \to \oo$, $u_n$ and $Du_n$ converge locally uniformly to respectively $u$ and $Du$, as can be seen from the method of characteristics.
	
	As a consequence, we see that, as $n \to \oo$, the vector field $b_n$ defined by
	\[
		b_n(t,x) := D_p H(x,Du_n(t,x),\phi_n) \quad \text{for } (t,x) \in [0,T] \times \RR^d
	\]
	converges locally uniformly to $b := D_p H(\cdot,Du,\phi)$. Also, by \eqref{ubounds} and Lemma \ref{L:classical}\eqref{L:continuity}, we have that $(m_n)_{n=1}^\oo$ is uniformly bounded in $C([0,T],\mcl P_\lambda)$, and is therefore relatively compact in the weak-$*$ topology. The result then follows upon showing that the only weak-$*$ limit point of $(m_n)_{n=1}^\oo$ is $m$. Indeed, given $\rho \in C^1([0,T] \times \RR^d)$ with compact support, we have, for all $t \in[0,T]$,
	\[
		\int_{\RR^d} \rho(t,x) m_n(t)(dx) = \int_{\RR^d} \rho(0,x) m_0(dx) + \int_0^t \int_{\RR^d} \left[ \del_t \rho(s,x) - b_n(s,x) \cdot D_x \rho(s,x) \right] m_n(s)(dx),
	\]
	and so every limit point of $m_n$ is a weak solution of \eqref{E:givenuphi}, and hence equal to $m$.
\end{proof}

We now present the existence result.

\begin{theorem}\label{T:controlsexistence}
	Assume \eqref{A:powers} - \eqref{A:AandB}.Then there exists a classical solution of \eqref{E:reducedcontrols}.
\end{theorem}

\begin{proof}
	By Lemma \ref{L:goodT}, the map $\mcl T$ is continuous from $\mcl P_\gamma$ into a bounded subset of $\mcl P_\lambda$. Since $\lambda > \gamma$, bounded sets of $\mcl P_\lambda$ are precompact in $\mcl P_\gamma$ with respect to the $d_\gamma$-metric, and therefore, by Schauder's fixed point theorem, $\mcl T$ has a fixed point. It is then clear that the resulting triple $(u,m,\phi)$ is a classical solution of \eqref{E:reducedcontrols}.
\end{proof}

\subsection{Examples}

We consider a general example in which the dependence of the Hamiltonian on the measure $\mu$ is through an affine drift term:

\begin{equation}\label{E:drift}
	\begin{dcases}
		-\frac{\del u}{\del t} + H(Du) - b(x,\phi)\cdot Du = 0 \quad \text{in } \RR^d \times [0,T], & u(\cdot,T) = G(x,m) \\
		\frac{\del m}{\del t} - \DIV\left[ (DH(Du) - b(x,\phi))m \right] = 0 \quad \text{in } \RR^d \times [0,T], & m(\cdot,0) = m_0,\\
		\phi(t) := \int_{\RR^d} \Phi(Du(t,y))m_t(dy) & \text{for } t \in [0,T].
	\end{dcases}
\end{equation}
We assume that
\begin{equation}\label{A:Phihomog}
	\Phi \in C^1(\RR^d,\RR^m) \quad \text{and, for some matrix $\mcl A \in \RR^{m\times m}$, } D\Phi(p) \cdot p = \mcl A\Phi(p) \quad \text{for all } p \in \RR^d,
\end{equation}
and we assume that $b$ is affine in $x$, that is, for some functions $a \in C^{0,1}(\RR^m,\RR)$ and $b \in C^{0,1}(\RR^m ,\RR^d)$, $F$ takes the form
\[
	b(x,\phi) = a(\phi)x + b(\phi).
\]
A straightforward computation then shows that \eqref{A:controlreduction} is satisfied with
\[
	A(t,\phi) := a(\phi)\mcl A \quad \text{and} \quad B(t,\phi) \equiv 0.
\]
If we also assume that $G$ satisfies \eqref{A:Ggrowth} and
\[
	\left\{
	\begin{split}
	&H \in C^{1,1}(\RR^d) \text{ is convex and, for some } C > 0,\\
	&\sup_{p \in \RR^d} \pars{ \frac{|H(p)|}{1 + |p|^q} + \frac{|DH(p)|}{1 + |p|^{q-1}} } < \oo,
	\end{split}
	\right.
\]
then all of the assumptions of Theorem \ref{T:controlsexistence} are satisfied, and therefore, by Proposition \ref{P:reducedcontrols}, there exists a classical solution of \eqref{E:drift}.

With some further structure, the entire system \eqref{E:drift} reduces to a system of ordinary differential equations. We consider a single spatial dimension for simplicity, and, for $p > 1$, we study the system
\begin{equation}\label{E:controlspower}
	\begin{dcases}
	-u_t + \frac{1}{p'} \pars{ \frac{1}{p}|u_x|^p - xa(\phi) u_x }= 0, & u(T,x) = \frac{1}{p'} |x|^{p'} g(z),\\
	m_t -  \frac{1}{p'} \left[ m \pars{ u_x|u_x|^{p-2} - xa(\phi)} \right]_x = 0, & m(0,x) = m_0(x),\\
	z(t) = \frac{1}{p'} \int_{\RR} |y|^{p'} m(y,t)dy, &\\
	\phi(t) = \frac{1}{q} \int_\RR |u_y(y,t)|^q m(y,t)dy. &
	\end{dcases}
\end{equation}
Above, we assume that
\begin{equation}\label{A:gmonotonepositive}
	g: \RR \to \RR \quad \text{is Lipschitz, nonnegative, and nondecreasing},
\end{equation}
and
\begin{equation}\label{A:aLip}
	a: \RR \to \RR \quad \text{is Lipschitz and bounded.}
\end{equation}

We can reduce by formally setting
\[
	u(t,x) = \frac{1}{p'}|x|^{p'} \psi(t).
\]
This leads to the system
\begin{equation}\label{E:reducedcontrolseg}
	\begin{dcases}
		- \dot \psi + \frac{1}{p}|\psi|^p - a(\phi) \psi = 0 & \psi(T) = g(z(T)),\\
		\dot z + z \pars{\psi |\psi|^{p-2} - a(\phi)} = 0 & z(0) = z_0,\\
		\dot \phi + \frac{q}{p'} a(\phi)\phi = 0 & \phi(0) = \alpha_0 |\psi(0)|^q,
	\end{dcases}
\end{equation}
where
\[
	z_0 := \frac{1}{p'} \int |y|^{p'} m_0(dy) \quad \text{and} \quad \alpha_0 := \frac{1}{q} \int |x|^{q(p'-1)}m_0(dx).
\]
The results in the previous sub-section then guarantee the existence of a solution of \eqref{E:reducedcontrolseg}. Notice also that, for some constants $0 < c_0 < C_0$ depending only on bounds for $g$, $z_0$, and $\alpha_0$, any solution satisfies
\begin{equation}\label{psizapriori}
	c_0 \le \psi(t) \le C_0 \quad \text{and} \quad c_0 \le z(t) \le C_0 \quad \text{for all } t \in [0,T].
\end{equation}

We finish by looking at the particular case in which $p = q$, and we present a uniqueness result with an additional assumption. 

If $p=q$, then, in fact, for all $t \in [0,T]$,
\[
	\phi(t) = z(t) |\psi(t)|^q.
\]
In that case, the solvability of \eqref{E:reducedcontrolseg} reduces to a system involving only $z$ and $\psi$, namely
\begin{equation}\label{E:pequalsq}
	\begin{dcases}
		- \dot \psi + \frac{1}{p}|\psi|^p - a(z|\psi|^p) \psi = 0 & \psi(T) = g(z(T)),\\
		\dot z + z \pars{\psi |\psi|^{p-2} - a(z|\psi|^p)} = 0 & z(0) = z_0.
	\end{dcases}
\end{equation}
We now introduce the assumption that
\begin{equation}\label{A:smalla}
	\left\{
	\begin{split}
	&\text{there exist $0 < \delta_0 < \delta_1$, such that}\\
	&\delta_0 < \phi^{1/p} a'(\phi) \le \delta_1 \quad \text{for all } \phi > 0.
	\end{split}
	\right.
\end{equation}

\begin{proposition}\label{P:pequalsq}
	Assume \eqref{A:gmonotonepositive}, \eqref{A:aLip}, and \eqref{A:smalla}. Then, if $\delta_1 > 0$ is sufficiently small, the solution of \eqref{E:pequalsq} is unique.
\end{proposition}

\begin{proof}
	To simplify the notation, let us write, for $\psi,z \in \RR_+$,
	\[
		A(\psi,z) := \frac{1}{p}\psi^p - a \pars{ z \psi^p}\psi \quad \text{and} \quad B(\psi,z) = \psi^{p-1} - a(z\psi^p),
	\]
	so that \eqref{E:pequalsq} can be written as
	\begin{equation}\label{E:pequalsqsimple}
	\begin{dcases}
		- \dot \psi + A(\psi,z) = 0 & \psi(T) = g(z(T)),\\
		\dot z + z B(\psi,z) = 0 & z(0) = z_0.
	\end{dcases}
	\end{equation}
	Given two solutions $(\psi_1,z_1)$ and $(\psi_2,z_2)$ of \eqref{E:pequalsqsimple}, define
	\[
		X(t) = (\psi_1(t) - \psi_2(t))(z_1(t) - z_2(t)) \quad \text{for } t \in [0,T].
	\]
	For $\tau \in [0,1]$, define also
	\[
		\psi^\tau := \tau \psi_1 + (1-\tau)\psi_2 \quad \text{and} \quad z^\tau := \tau z_1 + (1-\tau) z_2.
	\]
	Then
	\begin{align*}
		\dot X(t) &= \pars{ A(\psi_1,z_1) - A(\psi_2,z_2)}(z_1-z_2) - (\psi_1 - \psi_2)\pars{ z_1B(\psi_1,z_1) - z_2 B(\psi_2,z_2)}\\
		&= - \int_0^1 z^\tau B_\psi(\psi^\tau,z^\tau) d\tau (\psi_1 - \psi_2)^2 \\
		&+ \int_0^1 \pars{ A_\psi(\psi^\tau,z^\tau) - B(\psi^\tau,z^\tau) - z^\tau B_z(\psi^\tau,z^\tau) }d\tau (\psi_1 - \psi_2)(z_1-z_2) \\
		&+ \int_0^1 A_z(\psi^\tau,z^\tau)d\tau (z_1-z_2)^2.
	\end{align*}
	We claim that, if $\delta_1$ in \eqref{A:smalla} is sufficiently small, then, for some constant $\eps_0 > 0$,
	\[
		\dot X(t) \le -\eps_0\pars{ (\psi_1(t) - \psi_2(t) )^2 + (z_1(t) - z_2(t))^2 },
	\]
	which will follow in turn from the strict negativity of matrix
	\[
		\mcl M(\psi,z) := 
		\begin{pmatrix}
			-zB_\psi(\psi,z) & \frac{1}{2} \pars{ A_\psi(\psi,z) - B(\psi,z) - zB(\psi,z)} \\
			\frac{1}{2} \pars{ A_\psi(\psi,z) - B(\psi,z) - zB(\psi,z)} & A_z(\psi,z)
		\end{pmatrix}
	\]
	for all $\psi,z$ satisfying \eqref{psizapriori}.
	We note that
	\[
		A_z(\psi,z) = -a'(z\psi^p)\psi^{p+1} \le -\delta_0 z^{-1/p} \psi^p \le - \delta_0 \frac{c_0^p}{C_0^{1/p}}
	\]
	and
	\begin{align*}
		zB_\psi(\psi,z) &= z\psi^{p-2}\left[ p-1 - p z \psi a'(z\psi^p) \right]\\
		&\ge z \psi^{p-2} \left[ p-1 - \delta_1 p z^{1 - 1/p} \right] \\
		&\ge z \psi^{p-2} \left[ p-1 - C_0^{1 - 1/p} p \delta_1 \right],
	\end{align*}
	and, therefore, $\tr \mcl M(\psi,z) < 0$ as long as
	\[
		\delta_1 \le \frac{p-1}{pC_0^{1 - 1/p}}.
	\]
	We also compute
	\[
		A_\psi(\psi,z) - B(\psi,z) - zB(\psi,z) = -(p-1)z \psi^p a'(z\psi^p),
	\]
	and, hence,
	\begin{align*}
		\det \mcl M(\psi,z) &= a'(z\psi^p)z\psi^{2p-1}\left[ p-1 - p z \psi a'(z\psi^p) \right] - \frac{(p-1)^2}{4}z^2 \psi^{2p}a'(z\psi^p)^2\\
		&= a'(z\psi^p)z\psi^{2p-1} \left[ p-1 - \pars{ p + \frac{(p-1)^2}{4} } z \psi a'(z\psi^p) \right]\\
		&\ge a'(z\psi^p)z\psi^{2p-1} \left[ p-1 - \delta_1 \pars{ p + \frac{(p-1)^2}{4}} z^{1-1/p} \right] \\
		&\ge a'(z\psi^p)z\psi^{2p-1} \left[ p-1 - C_0^{1 - 1/p} \delta_1 \pars{ p + \frac{(p-1)^2}{4}} \right].
	\end{align*}
	Further restricting $\delta_1$ yields
	\[
		\det \mcl M(\psi,z) > 0,
	\]
	and we conclude that the largest eigenvalue $\mcl M(\psi,z)$ is negative and bounded away from zero.
	
	As a consequence, we find that, using the fact that $z_1(0) = z_2(0) = z_0$,
	\begin{align*}
		\pars{g(z_1(T)) - g(z_2(T))}(z_1(T) - z_2(T)) &= X(T) - X(0) = \int_0^T \dot X(t)dt \\
		&\le -\eps_0 \int_0^T \pars{ (\psi_1(t) - \psi_2(t))^2 + (z_1(t) - z_2(t))^2 }dt.
	\end{align*}
	By \eqref{A:gmonotonepositive}, we get $\psi_1 = \psi_2$ and $z_1 = z_2$, as desired.
\end{proof}

\section{Small noise expansions}\label{S:smallnoise}

We return to the setting of the generalized master equation in a finite state space as in Section \ref{S:finite}, and add a small noise term, which, in general, breaks the algebraic conditions that allow for a reduced problem. This motivates comparing the solution of the noisy mean field games problem to a formal small noise expansion.

\subsection{A stability result}

For smooth, monotone maps $(G,F): \RR^{2N} \to \RR^{2N}$ and $U_0 : \RR^N \to \RR^N$, an affine map $\mcl T: \RR^N \to \RR^N$, and $\lambda > 0$, we consider the equation
\begin{equation}\label{E:finitestatenoise}
	\del_t U + \left[ F(x,U) \cdot \nabla_x \right]U + \lambda \pars{ U(t,x) - \mcl T^* U(t,\mcl Tx)} = G(x,U) \quad \text{in } \RR^N \times (0,T], \quad U(0,\cdot) = U_0.
\end{equation}
This is similar to the equation \eqref{E:finitestate} from Section \ref{S:finite}. The interpretation of the term involving $\mcl T$ is that, at random times with an exponential law of parameter $\lambda$, the players are rearranged by the map $\mcl T$.

The inclusion of the map $\mcl T$ as in \eqref{E:finitestatenoise} is just one of many ways to model a common noise effect in finite state space mean field games. Other examples can be found in \cite{BLL}. 

We assume a stricter sort of monotonicity for $G$; in particular, for some $\alpha > 0$,
\begin{equation}\label{A:GFalphamonotone}
	\langle G(x,U) - G(y,V),x-y \rangle + \langle F(x,U) - F(y,V),U-V \rangle \ge \alpha|x-y|^2 \quad \text{for all }(x,y,U,V) \in \RR^{4N}.
\end{equation}
Under the above assumptions, \eqref{E:finitestatenoise} has a unique classical solution, and moreover \cite{BLL,BLLplanning}, there exists a constant $L_0 > 0$ depending only on $\alpha$, $T$, $\nor{D_x F}{\oo}$, $\nor{D_x G}{\oo}$, $\nor{DU_0}{\oo}$, $\norm{\mcl T}$, and an upper bound for $\lambda$ such that
\begin{equation}\label{Lipschitzbound}
	\sup_{t \in [0,T]} \nor{D_x U(t,\cdot)}{\oo} \le L_0.
\end{equation}

We next suppose that, for some smooth and bounded $R: [0,T] \times \RR^N \to \RR^N$, $V:[0,T] \times \RR^N \to \RR^N$ satisfies
\begin{equation}\label{E:finitestatenoiseperturbed}
	\del_t V + \left[ F(x,V) \cdot \nabla_x \right]V + \lambda \pars{ V(t,x) - \mcl T^* V(t,\mcl Tx)} = G(x,V) + R(t,x) \quad \text{in } \RR^N \times (0,T], \quad V(0,\cdot) = U_0.
\end{equation}
We then have the following result:

\begin{theorem}\label{T:stability}
	Assume \eqref{A:GFalphamonotone} and that $U$ and $V$ solve respectively \eqref{E:finitestatenoise} and \eqref{E:finitestatenoiseperturbed}. Then
	\[
		\sup_{(t,x) \in [0,T] \times \RR^N} \abs{ U(t,x) - V(t,x)} \le \pars{ \frac{L_0 T}{\alpha}}^{1/2}\nor{R}{\oo}.
	\]
\end{theorem}

\begin{proof}
	Define
	\[
		W(t,x,y) = \langle U(t,x) - V(t,y), x-y \rangle
	\]
	which satisfies the equation
	\begin{equation}\label{E:W}
		\left\{
		\begin{split}
			&\del_t W + F(x,U) \cdot \nabla_x W + F(y,V) \cdot \nabla_y W + \lambda \pars{ W - W(\cdot,\mcl Tx,\mcl Ty)} \\
			&\quad = \langle G(x,U) - G(y,V),x-y \rangle + \langle F(x,U) - F(y,V), U - V \rangle - \langle R(t,x),x-y \rangle, \\
			&W(0,x,y) = \langle U_0(x) - U_0(y),x-y \rangle.
		\end{split}
		\right.
	\end{equation}
	In view of \eqref{A:GFalphamonotone}, the right-hand side satisfies
	\begin{align*}
		 &\langle G(x,U) - G(y,V),x-y \rangle + \langle F(x,U) - F(y,V), U - V \rangle - \langle R(t,x),x-y \rangle\\
		 &\ge \alpha |x-y|^2 - \nor{R}{\oo}|x-y| \ge - \frac{\nor{R}{\oo}^2}{4 \alpha}.
	\end{align*}
	As a consequence, the maximum principle implies that, for all $(t,x,y) \in [0,T] \times \RR^N \times \RR^N$,
	\[
		\langle U(t,x) - V(t,y), x-y \rangle \ge - \frac{\nor{R}{\oo}^2T}{4 \alpha}.
	\]
	Fix $\delta > 0$ and $\xi \in B_1$, and let $x = y + \delta \xi$. Then this inequality becomes
	\[
		\langle U(t,y + \delta \xi) - V(t,y), \xi \rangle \ge - \frac{ \nor{R}{\oo}^2T}{4 \alpha \delta}.
	\]
	Using the Lipschitz bound \eqref{Lipschitzbound}, we then have
	\[
		\langle U(t,y) - V(t,y), \xi \rangle \ge - \frac{\nor{R}{\oo}^2T}{4 \alpha \delta} - L_0 \delta,
	\]
	and so, because $\xi$ is arbitrary,
	\[
		\nor{U - V}{\oo} \le \frac{\nor{R}{\oo}^2T}{4 \alpha \delta} + L_0 \delta.
	\]
	The estimate is optimized, and the result proved, upon choosing
	\[
		\delta := \pars{\frac{T}{4 \alpha L_0} }^{1/2} \nor{R}{\oo}.
	\]
\end{proof}

\subsection{Small noise expansions}

We apply the stability estimate from the previous sub-section in order to prove some small-noise expansion results. 

We consider, for some $\eps \in (0,1)$, the equation
\begin{equation}\label{E:finitestatesmallnoise}
	\del_t U^\eps + \left[ F(x,U^\eps) \cdot \nabla_x \right]U^\eps + \eps \pars{ U^\eps(t,x) - \mcl T^* U^\eps(t,\mcl Tx)} = G(x,U^\eps) \quad \text{in } \RR^N \times (0,T], \quad U^\eps(0,\cdot) = U_0
\end{equation}
as well as its deterministic counterpart
\begin{equation}\label{E:finitestatenonoise}
	\del_t U + \left[ F(x,U) \cdot \nabla_x \right]U = G(x,U) \quad \text{in } \RR^N \times (0,T], \quad U(0,\cdot) = U_0.
\end{equation}
We also consider the solution $V: [0,T] \times \RR^N \to \RR^N$ of the linearized problem
\begin{equation}\label{E:linearized}
	\left\{
	\begin{split}
	&\del_t V + \left[ F(x,V) \cdot \nabla \right] U + \pars{ \nabla_u F(x,U) \cdot \nabla U - \nabla_u G(x,U)} V + U - T^*U(t,Tx) = 0 \quad \text{in } \RR^N \times [0,T],\\
	&V(0,x) = 0 \quad \text{in } \RR^N.
	\end{split}
	\right.
\end{equation}
Formally, the first-order expansion
\[
	V^\eps(t,x) := U(t,x) + \eps V(t,x)
\]
approximates $U^\eps$ to an error of order better than $\eps$, which we make precise with the following result.

\begin{theorem}\label{T:smallnoise}
	There exists a constant $C > 0$ depending only on $\alpha$, $T$, $\nor{DF}{C^2}$, $\nor{DG}{C^2}$, $\nor{DU_0}{\oo}$, and $\norm{\mcl T}$ such that
	\[
		\sup_{(t,x) \in [0,T] \times \RR^N} \abs{U^\eps(t,x) - V^\eps(t,x)} \le C \eps^2.
	\]
\end{theorem}

\begin{proof}
	Routine computations reveal that there exists $R^\eps: [0,T] \times \RR^N  \to \RR^N$ such that, for some constant $C > 0$ depending only on $\nor{D^2 F}{\oo}$ and $\nor{D^2G}{\oo}$,
	\[
		\sup_{(t,x) \in \RR^N} |R^\eps(t,x)| \le C\eps^2
	\]
	and
	\begin{equation}\label{E:finitestatenoiseapprox}
		\left\{
		\begin{split}
		&\del_t V^\eps + \left[ F(x,V^\eps) \cdot \nabla_x \right]V^\eps + \eps \pars{ V^\eps(t,x) - \mcl T^* V^\eps(t,\mcl Tx)} = G(x,V^\eps) + R^\eps(t,x) \quad \text{in } \RR^N \times (0,T], \\
		&V^\eps(0,\cdot) = U_0.
		\end{split}
		\right.
	\end{equation}
	The result now follows from Theorem \ref{T:stability}.
\end{proof}

We finish this section by relating the above discussion to the reduction phenomena observed in Section \ref{S:finite}. In particular, it is unreasonable to expect that, if $L$ is the linear map defined there, then, for some $\tilde T: \RR^n \to \RR^n$ such that
\[
	LT = \tilde T L.
\]
This implies that the noisy map $T$ maps fibers of $L$ to fibers of $L$, and, in most practical applications, there is no reason to expect that the source of common noise interacts so well with the reduced quantities. In this case, even though \eqref{E:finitestatesmallnoise} does not fully reduce, Theorem \ref{T:smallnoise} gives a way to relate, up to small error, the reduced version of \eqref{E:finitestatenonoise} to the solution of \eqref{E:finitestatesmallnoise}.

\bibliography{reducedMFG}{}
\bibliographystyle{acm}

\end{document}